\documentclass[letterpaper,11pt]{amsart}

\usepackage[margin=1.2in]{geometry}
\usepackage{amsmath,amsthm,amssymb}
\usepackage{xspace,xcolor}
\usepackage[breaklinks,colorlinks,citecolor=teal,linkcolor=teal,urlcolor=teal,pagebackref,hyperindex]{hyperref}
\usepackage[alphabetic]{amsrefs}
\usepackage[all]{xy}
\usepackage{color}

\theoremstyle{plain}
\newtheorem{thm}{Theorem}[section]

\newtheorem{lem}[thm]{Lemma}
\newtheorem{prop}[thm]{Proposition}
\newtheorem{cor}[thm]{Corollary}
\newtheorem{example}[thm]{Example}

\theoremstyle{definition}

\theoremstyle{remark}
\newtheorem{rmk}[thm]{Remark}

\def\C{{\mathbf C}}

\def\A{{\mathbf A}}

\def\cD{\mathcal{D}}

\def\cH{\mathcal{H}}

\def\cJ{\mathcal{J}}

\def\cO{\mathcal{O}}

\def\frm{\mathfrak{m}}

\def\.{\cdot}
\def\^{\widehat}

\def\({\left(}
\def\){\right)}

\newcommand{\llbracket}{[\negthinspace[}
\newcommand{\rrbracket}{]\negthinspace]}

\renewcommand{\and}{ \ \ \text{ and } \ \ }

\usepackage{soul}

\begin{document}

\author{Mircea Musta\c{t}\u{a}}

\address{Department of Mathematics, University of Michigan, 530 Church Street, Ann Arbor, MI 48109, USA}

\email{mmustata@umich.edu}

\author{Sebasti\'{a}n Olano}

\address{Department of Mathematics, University of Michigan, 530 Church Street, Ann Arbor, MI 48109, USA}

\email{olano@umich.edu}

\thanks{M.M. was partially supported by NSF grants DMS-2001132 and DMS-1952399.}


\title{On a conjecture of Bitoun and Schedler}

\begin{abstract}
Suppose that $X$ is a smooth complex algebraic variety of dimension $\geq 3$ and $f$ defines a hypersurface $Z$ in $X$, with a unique singular point $P$.
Bitoun and Schedler conjectured that the $\cD$-module generated by $\tfrac{1}{f}$ has length equal to $g_P(Z)+2$, where $g_{P}(Z)$
is the reduced genus of $Z$ at $P$. We prove that this length is always $\geq g_P(Z)+2$ and equality holds if and only if $\tfrac{1}{f}$ lies
in the $\cD$-module generated by $I_0(f)\tfrac{1}{f}$, where $I_0(f)$ is the multiplier ideal $\cJ(f^{1-\epsilon})$, with $0<\epsilon\ll 1$.
In particular, we see that the conjecture holds if the pair $(X,Z)$ is log canonical.
We can also recover, with an easy proof, the result of Bitoun and Schedler saying that the conjecture holds for weighted homogeneous
isolated singularities. On the other hand, we give an example (a polynomial in $3$ variables with an ordinary singular point of multiplicity $4$)
for which the conjecture does not hold.
\end{abstract}

\maketitle

\section{Introduction}

Let $X$ be a smooth, irreducible, complex algebraic variety of dimension $n\geq 3$, and $Z$ an irreducible and reduced hypersurface in $X$ defined by $f\in\cO_X(X)$. 
We assume that $P\in Z$ is a point such that $Z\smallsetminus\{P\}$ is smooth. Recall that the localization 
$\cO_X(*Z):=\cO_X[1/f]$
has a natural structure of
left $\cD_X$-module, where $\cD_X$ is the sheaf of differential operators on $X$. In fact, $\cO_X(*Z)$ is a holonomic $\cD_X$-module; as such, it has finite length
in the category of $\cD_X$-modules (and the same property holds for all its $\cD_X$-submodules).

Bitoun and Schedler proposed in \cite{BS} a conjecture describing the length $\ell\big(\cD_X\cdot\tfrac{1}{f}\big)$
of the submodule $\cD_X\cdot\tfrac{1}{f} \subseteq \cO_X(*Z)$
 in terms of an invariant of $(Z,P)$, the reduced genus.  If $\varphi\colon Z'\to Z$ is a log resolution of $(Z,P)$ that is an isomorphism over
$Z\smallsetminus\{P\}$ and if $E=\varphi^{-1}(P)_{\rm red}$, then the reduced genus of $(Z,P)$ is $g_P(Z):=h^{n-2}(E,\cO_E)=h^0(E,\omega_E)$. With this notation, Bitoun and Schedler
conjectured that  $\ell\big(\cD_X\cdot\tfrac{1}{f}\big)=g_P(Z)+2$ and they proved the conjecture in the case when $f\in\C[x_1,\ldots,x_n]$ is a weighted
homogeneous polynomial. 

Recall now that for every $\lambda>0$, one can associate to $f$ the multiplier ideal $\cJ(f^{\lambda})$ of exponent $\lambda$ (see \cite[Chapter~9]{Lazarsfeld} for
an introduction to multiplier ideals). We put $I_0(Z)=\cJ(f^{1-\epsilon})$, where $0<\epsilon\ll 1$.
The following is our main result:

\begin{thm}\label{thm_intro}
With the above notation, we always have 
$$\ell\big(\cD_X\cdot\tfrac{1}{f}\big)\geq g_P(Z)+2.$$
Moreover, equality holds if and only if $\tfrac{1}{f}$ lies in the $\cD_X$-submodule of $\cO_X(*Z)$ generated by $I_0(Z)\tfrac{1}{f}$. 
\end{thm}

Note that $I_0(Z)=\cO_X$ if and only if the pair $(X,Z)$ is log canonical, hence we obtain

\begin{cor}\label{cor_intro}
With the above notation, if the pair $(X,Z)$ is log canonical, then $\ell\big(\cD_X\cdot\tfrac{1}{f}\big)=g_P(Z)+2$.
\end{cor}

We note that $\cO_X(*Z)$ underlies a mixed Hodge module in the sense of Saito's theory \cite{Saito-MHM}. In particular, it carries a Hodge filtration $F_{\bullet}\cO_X(*Z)$ such that
$F_k\cO_X(*Z)=0$ for $k<0$ and $F_0\cO_X(*Z)=I_0(Z)\frac{1}{f}$. In general, it is known that this Hodge filtration is contained in the pole order filtration, that is, we have
$$F_k\cO_X(*Z)\subseteq P_k\cO_X(*Z):=\cO_X\tfrac{1}{f^{k+1}}\quad\text{for all}\quad k\geq 0,$$
with equality if $Z$ is smooth 
(these results have been proved by Saito in \cite{Saito-B} and \cite{Saito-HF}). With this terminology, Theorem~\ref{thm_intro} says that the conjecture of Bitoun and Schedler holds for $Z$ if and only if 
$P_0\cO_X(*Z)$ and $F_0\cO_X(*Z)$ generate the same $\cD_X$-submodule of $\cO_X(*Z)$. 
For weighted homogeneous isolated singularities, we prove the following stronger result (see Section~\ref{section_weighted_homogeneous}
for the definition of weighted homogeneous singularities):

\begin{thm}\label{thm_quasi_homog}
If $X$ is a smooth, irreducible, complex algebraic variety of dimension $n\geq 2$ and $Z$ is
a hypersurface in $X$ defined by $f\in\cO_X(X)$, which has weighted homogeneous isolated singularities, then
for every $k\geq 0$, $F_k\cO_X(*Z)$ and $P_k\cO_X(*Z)$ generate the same $\cD_X$-submodule
of $\cO_X(*Z)$.
\end{thm}

In particular, by taking $n\geq 3$ and $k=0$ and using also Theorem~\ref{thm_intro}, we 
recover the main result in \cite{BS}, saying that the conjecture holds for weighted homogeneous
isolated singularities.

On the other hand, we give a counterexample to the Bitoun-Schedler conjecture: we show that it fails for $f=x^4+y^4+z^4+xy^2z^2$, by proving that the property in 
Theorem~\ref{thm_intro} does not hold in this case (see Proposition~\ref{prop_example}). In order to show this, we exploit the fact that this is a semi-quasi-homogeneous
singularity and use the description of the Hodge filtration on $\cO_X(*Z)$ from \cite{Saito-HF}*{Theorem 0.9}. 

Finally, since a previous version of this paper was made public, there has been further work on the Bitoun-Schedler conjecture: Saito 
gave in \cite{Saito-recent} an interpretation of $\ell\big(\cD_X\cdot\tfrac{1}{f}\big)$ (and, more generally, of $\ell(\cD_X\cdot f^{-\alpha})$
for $\alpha\in {\mathbf Q}$) in terms of the Brieskorn lattice of $f$. Building on this, he gave a series of counterexamples to the conjecture,
extending the one described above.

\noindent
{\bf Outline and acknowledgment.}
The paper is organized as follows: in Section~2 we discuss the reduced genus of an isolated singularity and give a formula for this invariant in terms of the multiplier ideal $I_0(Z)$ 
and the adjoint ideal. We use this in Section~3 to prove Theorem~\ref{thm_intro}. In Section~4 we discuss weighted homogeneous singularities and prove 
Theorem~\ref{thm_quasi_homog}. Finally, in Section~5 we give the counterexample to the Bitoun-Schedler conjecture.

We are indebted to Uli Walther who explained to us how to approach the Macaulay 2 computation that first showed us that we had a counterexample to the conjecture. 
We thank Thomas Bitoun for his comments on a preliminary version of this note. We are also grateful to the anonymous referees for suggesting changes that improved the presentation of the paper.

\section{A formula for the reduced genus}

We begin by recalling some definitions concerning log resolutions and certain invariants of singularities that we will be using.
For details, we refer to \cite[Chapter~9]{Lazarsfeld}.

Given a complex algebraic variety $Z$ (always assumed to be reduced and irreducible) and a proper closed subscheme $Z'\hookrightarrow Z$ such that $Z\smallsetminus 
{\rm Supp}(Z')$
is smooth, a \emph{log resolution} of $(Z,Z')$ is a proper morphism $\varphi\colon \widetilde{Z}\to Z$ that is an isomorphism over $Z\smallsetminus {\rm Supp}(Z')$, such that $\widetilde{Z}$
is smooth and $\varphi^{-1}(Z')$ is an effective divisor with simple normal crossings. In particular, if $W$ is a proper closed subset of $Z$, viewed as a reduced closed subscheme, and if 
$Z\smallsetminus W$ is smooth, then we may consider log resolutions of $(Z,W)$. Log resolutions as above exist by Hironaka's fundamental theorem. 
Moreover, if $X$ is a smooth variety and $Z$ is a hypersurface in $X$,
then we may take a log resolution of $(X,Z)$ that is an isomorphism over $X\smallsetminus Z_{\rm sing}$, where $Z_{\rm sing}$ is the singular locus of $Z$.

Recall now that if $X$ is a smooth variety, $D$ is an effective divisor on $X$, and $\pi\colon \widetilde{X}\to X$ is a log resolution of $(X,D)$ with $F=\pi^*(D)$, then for every
$\lambda\in{\mathbf Q}_{>0}$, the multiplier ideal $\cJ(X,\lambda D)$ is defined by
$$\cJ(X,\lambda D)=\pi_*\cO_{\widetilde{X}}\big(K_{\widetilde{X}/X}-\lfloor \lambda F\rfloor\big).$$
Here $K_{\widetilde{X}/X}$ is the relative canonical divisor, the effective exceptional divisor locally defined by the determinant of the Jacobian matrix of $\pi$,
and for a ${\mathbf Q}$-divisor $G=\sum_ia_iG_i$, the round-down $\lfloor G\rfloor$ is given by $\sum_i\lfloor a_i\rfloor G_i$, 
where $\lfloor u\rfloor$ is the largest integer that is $\leq u$. 
We also put
$I_0(D):=\cJ\big(X,(1-\epsilon)D\big)$, for $0<\epsilon\ll 1$. Note that if $E=F_{\rm red}$, then
$$I_0(D)=\pi_*\cO_{\widetilde{X}}(K_{\widetilde{X}/X}-F+E).$$
The pair $(X,D)$ is log canonical if and only if $I_0(D)=\cO_X$.

Recall also that if $D$ is irreducible and reduced, and $\pi\colon\widetilde{X}\to X$ is a log resolution of $(X,D)$ as above, then the adjoint ideal ${\rm adj}(D)$ is defined by
$${\rm adj}(D)=\pi_*\cO_{\widetilde{X}}(K_{{\widetilde{X}}/X}-F+\widetilde{D}),$$
where $\widetilde{D}$ is the strict transform of $D$ on $\widetilde{X}$. 
Note that the inclusion ${\rm adj}(D)\subseteq I_0(D)$ always holds since $E-\widetilde{D}$ is effective.
We have ${\rm adj}(D)=\cO_X$ if and only if $D$ has rational singularities
(see \cite[Proposition~9.3.48]{Lazarsfeld}). 

We next recall the notion of reduced genus of a variety with isolated singularities. Suppose that $Z$ is a complex algebraic variety 
of dimension $n-1\geq 2$. We assume that $P\in Z$ is a point such that $Z\smallsetminus\{P\}$ is smooth. Let $\varphi\colon \widetilde{Z}\to Z$ be a log resolution of $(Z,P)$,
with $E=\varphi^{-1}(P)_{\rm red}$. In this case, the \emph{reduced genus} $g_P(Z)$ is $h^{n-2}(E,\cO_E)$. Note that $E$ is a proper scheme over ${\mathbf C}$, hence $g_P(Z)<\infty$. 
By Serre duality, we have $g_P(Z)=h^0(E,\omega_E)$ and since $E$ is a divisor on the smooth variety $\widetilde{Z}$,
the dualizing sheaf $\omega_E$ is isomorphic to $\omega_{\widetilde{Z}}(E)\vert_E$. 

\begin{lem}\label{lem_independence}
With the above notation, for every $i\geq 0$, the invariant $h^i(E,\cO_E)$ is independent of the choice of log resolution.
In particular, this is the case for $g_P(Z)$.
\end{lem}

\begin{proof}
Since any two log resolutions of $(Z,P)$ can be dominated by a third one, it follows that it is enough to show that if $\pi\colon \widetilde{Y}\to Y$ is a proper morphism,
with both $\widetilde{Y}$ and $Y$ smooth, 
$E$ is a reduced effective simple normal crossing divisor on $Y$ which is proper over ${\mathbf C}$, and $\pi$ is an isomorphism over $Y\smallsetminus E$, and 
$F=\pi^*(E)_{\rm red}$ has simple normal crossings, then $h^{i}(F,\cO_F)=h^{i}(E,\cO_E)$ for all $i$. By Serre duality, this is equivalent to showing that
$h^i(F,\omega_F)=h^i(E,\omega_E)$ for all $i$. 

Consider the short exact sequence on $\widetilde{Y}$:
\begin{equation}\label{eq1_lem_independence}
0\to \omega_{\widetilde{Y}}\to\omega_{\widetilde{Y}}(F)\to \omega_F\to 0.
\end{equation}
Note that $R^j\pi_*\omega_{\widetilde{Y}}=0$ for all $j\geq 1$ by Grauert-Riemenschneider vanishing and 
$R^j\pi_*\omega_{\widetilde{Y}}(F)=0$ for $j\geq 1$ by the Relative Vanishing theorem (see \cite[Theorem~9.4.1]{Lazarsfeld}). 
By taking the long exact sequence for direct images associated to (\ref{eq1_lem_independence}), we conclude that
\begin{equation}\label{eq5_lem_independence}
R^j\pi_*\omega_F=0\quad\text{for all}\quad j\geq 1
\end{equation}
and both rows in the following commutative diagram
$$\xymatrix{
0\ar[r] &\pi_*\omega_{\widetilde{Y}}\ar[r]\ar[d]_{\alpha} &\pi_*\omega_{\widetilde{Y}}(F)\ar[r]\ar[d]_{\beta} & \pi_*\omega_F\ar[r] \ar[d]_{\gamma} & 0\\
0\ar[r] & \omega_Y\ar[r] & \omega_Y(E)\ar[r] & \omega_E\ar[r] & 0
}$$
are exact.
Note that $\alpha$ is an isomorphism since $X$ is smooth: $\omega_{\widetilde{Y}}\simeq \pi^*\omega_Y(K_{\widetilde{Y}/Y})$ and
$\pi_*\cO_{\widetilde{Y}}(K_{\widetilde{Y}/Y})=\cO_Y$ since the divisor $K_{\widetilde{Y}/Y}$ is effective and exceptional. The morphism $\beta$
is an isomorphism too, due to the fact that $E$ has simple normal crossings: in terms of multiplier ideals, this says that $\cJ\big(Y, (1-\epsilon) E\big)=\cO_X$
for $0<\epsilon\ll 1$, which holds since the pair $(Y,E)$ is log canonical. We thus conclude that $\gamma$ is an isomorphism as well. 
By taking cohomology, we conclude that
$$H^i(E,\omega_E)\simeq H^i(E,\pi_*\omega_F)\simeq H^i(F,\omega_F),$$
where the second isomorphism follows from the Leray spectral sequence and the vanishings in (\ref{eq5_lem_independence}).
This completes the proof of the lemma. 
\end{proof}

In particular, we recover the following well-known

\begin{cor}\label{cor_indep}
If $Z$ is smooth, $P\in Z$, and $\pi\colon \widetilde{Z}\to Z$ is a log resolution of $(Z,P)$ with $f^{-1}(P)_{\rm red}=E$, then $h^i(E,\cO_E)=0$ for all $i>0$.
\end{cor}

\begin{proof}
By the lemma, the assertion is independent of the choice of log resolution. Since $Z$ is smooth, we may take $\pi$ to be the blow-up of $Z$ at $P$.
In this case $E$ is a projective space and the assertion in the corollary is clear.
\end{proof}

We next give a formula for the reduced genus in the case of hypersurface singularities.

\begin{prop}\label{formula_geom_genus}
Let $X$ be a smooth variety of dimension $n\geq 3$ and $Z\subset X$ a reduced and irreducible hypersurface. If $P\in Z$ is a point such that $Z\smallsetminus\{P\}$ is
smooth, then
$$g_P(Z)=\dim_{\mathbf C}\big(I_0(Z)/{\rm adj}(Z)\big).$$
\end{prop}

\begin{proof}
After possibly replacing $X$ by an affine open neighborhood of $P$, we may and will assume that $X$ is affine.
Let $\pi\colon \widetilde{X}\to X$ be a log resolution of $(X,Z)$ that is an isomorphism over $X\smallsetminus \{P\}$. We put $F=\pi^*(Z)$ and $E=F_{\rm red}$.
We also write $E=\widetilde{Z}+T$, where $\widetilde{Z}$ is the strict transform of $Z$ and $T$ is the reduced exceptional divisor. Note that the induced morphism
$\widetilde{Z}\to Z$ is a log resolution of $(Z,P)$, with reduced exceptional divisor $T\cap \widetilde{Z}$, hence 
$g_P(Z)=h^{n-2}(T\cap \widetilde{Z}, \cO_{T\cap\widetilde{Z}})$. 

On $\widetilde{X}$ we have the short exact sequence
\begin{equation}\label{eq1_formula_geom_genus}
0\to\cO_{\widetilde{X}}(K_{\widetilde{X}/X}-F+\widetilde{Z})\to \cO_{\widetilde{X}}(K_{\widetilde{X}/X}-F+E)\to
\cO_{\widetilde{X}}(K_{\widetilde{X}/X}-F+E)\vert_T\to 0.
\end{equation}
Note that 
$$
R^1\pi_*\cO_{\widetilde{X}}(K_{\widetilde{X}/X}-F+\widetilde{Z})=0.
$$
Indeed, this follows using the projection formula if we show that $R^1\pi_*\omega_{\widetilde{X}}(\widetilde{Z})=0$. 
This follows by taking the long exact sequence for higher direct images corresponding to the short exact sequence
$$0\to \omega_{\widetilde{X}}\to \omega_{\widetilde{X}}(\widetilde{Z})\to \omega_{\widetilde{Z}}\to 0,$$
using the fact that $R^1\pi_*\omega_{\widetilde{X}}=0$ and $R^1\pi_*\omega_{\widetilde{Z}}=0$ by Grauert-Riemenschneider vanishing.

By taking the long exact sequence for higher direct images corresponding to (\ref{eq1_formula_geom_genus}), we thus get a short exact sequence
\begin{equation}\label{eq2_formula_geom_genus}
0\to {\rm adj}(Z)\to I_0(Z)\to \pi_*\big(\cO_{\widetilde{X}}(K_{\widetilde{X}/X}-F+E)\vert_T\big)\to 0.
\end{equation}
Since $T$ lies above $P$, it follows that $\pi^*(\omega_X)\vert_T\simeq\cO_T$ and $\cO_{\widetilde{X}}(F)\vert_T\simeq\cO_T$. 
Moreover, the adjunction formula implies that 
$$\omega_{\widetilde{X}}(E)\vert_T\simeq\omega_T(\widetilde{Z}\vert_T).$$
We thus conclude that
$$\pi_*\big(\cO_{\widetilde{X}}(K_{\widetilde{X}/X}-F+E)\vert_T\big)\simeq H^0\big(T,\omega_T(\widetilde{Z}\vert_T)\big),$$
where the right-hand side is viewed as a skyscraper sheaf supported on $P$. Using the exact sequence (\ref{eq2_formula_geom_genus})
and Serre duality we thus conclude that
\begin{equation}\label{eq3_formula_geom_genus}
\dim_{\mathbf C}\big(I_0(Z)/{\rm adj}(Z)\big)=h^0\big(T,\omega_T(\widetilde{Z}\vert_T)\big)=h^{n-1}(T, \cO_T(-\widetilde{Z}\vert_T)\big).
\end{equation}
The short exact sequence on $T$:
$$0\to \cO_T\big(-\widetilde{Z}\vert_T\big)\to \cO_T\to \cO_{T\cap \widetilde{Z}}\to 0$$
gives an exact sequence
$$H^{n-2}(T,\cO_T)\to H^{n-2}(T\cap \widetilde{Z},\cO_{T\cap \widetilde{Z}})\to H^{n-1}\big(T, \cO_T(-\widetilde{Z}\vert_T)\big)\to H^{n-1}(T,\cO_T).$$
Since $n\geq 3$, we have
$$H^{n-2}(T,\cO_T)=0=H^{n-1}(T,\cO_T)$$
by Corollary~\ref{cor_indep}, hence the above exact sequence and (\ref{eq3_formula_geom_genus}) give
$$g_P(Z)=h^{n-2}(T\cap \widetilde{Z},\cO_{T\cap\widetilde{Z}})=h^{n-1}\big(T, \cO_T(-\widetilde{Z}\vert_T)\big)=\dim_{\mathbf C}\big(I_0(Z)/{\rm adj}(Z)\big).$$
\end{proof}

\begin{rmk}
If $Z$ is a hypersurface in a smooth variety $X$ of dimension $\geq 3$  and $Z$ has an isolated singularity at $P$, then in a neighborhood of $P$, $Z$ is reduced and irreducible.
Therefore, after replacing $X$ by a suitable neighborhood of $P$, we may always assume that $Z$ is reduced and irreducible.
\end{rmk}

\section{The proof of the main result}\label{sectionmainresult}

Let $X$ be a smooth complex algebraic variety. We denote by $\cD_X$ the sheaf of differential operators on $X$. For basic facts in the theory of
$\cD_X$-modules, we refer to \cite{HTT}. 

Let $Z$ be a hypersurface in $X$ defined by a nonzero $f\in\cO_X(X)$. If $j\colon U=X\smallsetminus Z\hookrightarrow X$ is the inclusion,
then the localization $\cO_X(Z):=j_*\cO_U=\cO_X[1/f]$ has a natural structure of $\cD_X$-module. In fact, it is a holonomic $\cD_X$-module 
(see \cite[Theorem~3.2.3]{HTT}). A basic fact is that every holonomic $\cD_X$-module has finite length; moreover, a $\cD_X$-submodule or quotient
module of a holonomic $\cD_X$-module has the same property
(for these facts, see \cite[Theorem~3.1.2]{HTT}). Therefore we may consider  
the length $\ell\big(\cD_X\cdot\tfrac{1}{f}\big)$ of the submodule of $\cO_X(*Z)$ generated by $\tfrac{1}{f}$. 

Note that inside $\cD_X\cdot\tfrac{1}{f}$ we have the irreducible $\cD_X$-submodule $\cO_X$, hence
$$\ell\big(\cD_X\cdot\tfrac{1}{f}\big)=\ell\big(\cD_X\cdot\tfrac{1}{f}/\cO_X\big)+1.$$
The quotient $\cD_X\cdot\tfrac{1}{f}/\cO_X$ is a $\cD_X$-submodule of the quotient $\cO_X(*Z)/\cO_X$, which is the local cohomology sheaf
$\cH_Z^1(\cO_X)$. We write $\big[\tfrac{1}{f}\big]$ for the class of $\tfrac{1}{f}$ in $\cH_Z^1(\cO_X)$.

Suppose from now on that $Z$ is reduced and irreducible. It is known that inside $\cH_Z^1(\cO_X)$ there is an irreducible $\cD_X$-module,
the intersection cohomology $\cD_X$-module of $Z$, that was introduced by Brylinski and Kashiwara \cite[Proposition~8.5]{BK}. We denote it by $M_f$.
This corresponds to the intersection cohomology complex of $Z$
via the Riemann-Hilbert correspondence. If $V=X\smallsetminus Z_{\rm sing}$, where $Z_{\rm sing}$ is the singular locus of $Z$, then
\begin{equation}\label{restriction_M_f}
M_f\vert_V=\cH_{V\cap Z}^1(\cO_V).
\end{equation}
In particular, this implies that the intersection of $M_f$ with $\cO_X\cdot \big[\tfrac{1}{f}\big]$ is nonzero. Since 
$M_f$ is an irreducible $\cD_X$-module, 
it follows that
$M_f\subseteq \cD_X\cdot\big[\tfrac{1}{f}\big]$. If we denote the quotient by $N_f$, using again the irreducibility of $M_f$, we conclude that
\begin{equation}\label{eq_N_f}
\ell\big(\cD_X\cdot\tfrac{1}{f}\big)=\ell(N_f)+2.
\end{equation}

In fact, the modules $\cO_X(*Z)$, $\cH_Z^1(\cO_X)$, and $M_f$ have more structure: they underlie mixed Hodge modules in the sense of Saito's theory
\cite{Saito-MHM}. In particular, they carry a Hodge filtration $F_{\bullet}$, which is an increasing filtration by coherent $\cO_X$-submodules, which is
compatible with the filtration on $\cD_X$ by order of differential operators. Any morphism of mixed Hodge modules is strict with respect to the Hodge filtration:
in particular, the Hodge filtration on $\cH_Z^1(\cO_X)$ is the quotient filtration induced from that on $\cO_X(*Z)$ and the Hodge filtration on $M_f$ is the submodule filtration
induced by that on $\cH^1_Z(\cO_X)$. 

It is known that $F_p\cO_X(*Z)=0$ for $p<0$ and $F_0\cO_X(*Z)=I_0(Z)\tfrac{1}{f}$ (see \cite[Theorem~0.4]{Saito-HF}). 
We thus have 
\begin{equation}\label{eq_F0}
F_0\cH^1_Z(\cO_X)=I_0(Z)\cdot \big[\tfrac{1}{f}\big]\subseteq\cD_X\cdot\big[\tfrac{1}{f}\big].
\end{equation}
In general, we have 
$$F_k\cO_X(*Z)\subseteq P_k\cO_X(*Z):=\cO_X\frac{1}{f^{k+1}} \quad\text{for all}\quad k\geq 0,$$ 
with equality if $Z$ is smooth (see \cite[Proposition~0.9]{Saito-B}).

On the other hand, we have 
\begin{equation}\label{eq_Mf}
F_0M_f={\rm adj}(Z)\cdot \big[\tfrac{1}{f}\big]
\end{equation}
by \cite[Theorem~A]{olano} (see also \cite[Proposition 3.5]{budur-saito}, \cite[Section 3.3]{budur}, and \cite[Theorem 0.6]{Saito-HF}). We can now prove our main result.

\begin{proof}[Proof of Theorem~\ref{thm_intro}]
Since $Z\smallsetminus\{P\}$ is smooth, it follows from (\ref{restriction_M_f}) that $N_f$ is supported on $\{P\}$. We deduce using
Kashiwara's equivalence (see \cite[Theorem~1.6.1]{HTT}) that $N_f\simeq (i_+\cO_{\{P\}})^{\oplus r}$, for some $r$, where $i\colon\{P\}\hookrightarrow X$
is the inclusion. In this case we have $\ell(N_f)=r$. Moreover, $r$ can be described as $r=\dim_{\mathbf C}N_f'$, where 
$N'_f=\{u\in N_f\mid\frm_Pu=0\}$, with $\frm_P$
being the ideal defining $P$. 

Note that we have an inclusion
\begin{equation}\label{inclusion}\big(F_0\cH^1_Z(\cO_X)\cap\cD_X\cdot\big[\tfrac{1}{f}\big]\big)/F_0M_f\hookrightarrow N_f.\end{equation}
Via (\ref{eq_F0}) and (\ref{eq_Mf}), the left-hand side is isomorphic as an $\cO_X$-module to
$I_0(Z)/{\rm adj}(Z)$. Note that this $\cO_X$-module is annihilated by $\frm_P$: this follows easily from the definition of the two ideals (see the short exact sequence
(\ref{eq2_formula_geom_genus}) in the proof of Proposition~\ref{formula_geom_genus}). 
We thus have an embedding
$$I_0(Z)/{\rm adj}(Z)\hookrightarrow N_f',$$
which gives 
$$\ell(N_f)=r=\dim_{\mathbf C}(N'_f)\geq\dim_{\mathbf C}\big(I_0(Z)/{\rm adj}(Z)\big)=g_P(Z),$$
where the last equality follows from Proposition~\ref{formula_geom_genus}. Using (\ref{eq_N_f}), we obtain 
$\ell\big(\cD_X\cdot\tfrac{1}{f}\big)\geq g_P(Z)+2$. Moreover, this is an equality if and only if $I_0(Z)/{\rm adj}(Z)=N'_f$. 

Since $N_f\simeq (i_+\cO_{\{P\}})^{\oplus r}$, it follows that $N_f$ is generated as a $\cD_X$-module by $N'_f$.
Moreover, an $\cO_X$-submodule $N''_f\subseteq N'_f$ generates $N_f$ over $\cD_X$ if and only if $N''_f=N'_f$.
We thus conclude that $\ell\big(\cD_X\cdot\tfrac{1}{f}\big)=g_P(Z)+2$ if and only if $N_f$ is generated over $\cD_X$
by the image of $I_0(Z)\cdot\big[\tfrac{1}{f}\big]$. In order to complete the proof of the theorem, it is enough to show that
this holds if and only if $\cD_X\cdot\tfrac{1}{f}\subseteq\cO_X(*Z)$ is generated over $\cD_X$ by $I_0(Z)\tfrac{1}{f}$. 
The ``if" part is clear, since $N_f$ is the quotient of $\cD_X\cdot\tfrac{1}{f}$ first by $\cO_X$, and then by $M_f$. 
The ``only if" part holds since $M_f\subseteq \cD_X\cdot I_0(Z)\tfrac{1}{f}$ (since $M_f$ is irreducible, it is contained in the $\cD_X$-submodule generated
by any nonzero subsheaf, such as ${\rm adj}(Z)\tfrac{1}{f}$) and $\cO_X\subseteq I_0(Z)\tfrac{1}{f}$ (this follows from the fact that $(f)=\cJ(Z)\subseteq I_0(Z)$). 
This completes the proof of the theorem.
\end{proof}

We can now deduce the assertion in the log canonical case.

\begin{proof}[Proof of Corollary~\ref{cor_intro}]
If $(X,Z)$ is log canonical, then $I_0(Z)=\cO_X$, hence it is clear that $\tfrac{1}{f}$ lies in $\cD_X\cdot I_0(Z)\tfrac{1}{f}$.
The assertion then follows from Theorem~\ref{thm_intro}.
\end{proof}

\section{The case of weighted homogeneous singularities}\label{section_weighted_homogeneous}

In this section we treat the case of weighted homogeneous singularities and prove Theorem~\ref{thm_quasi_homog}.
Let $X$ be a smooth complex algebraic variety and $Z$ a hypersurface on $X$ defined by $f\in\cO_X(X)$. 
Recall that $f$ is \emph{weighted homogeneous} at $P\in Z$ if there is
a regular system of parameters $x_1,\ldots,x_n$ in $\cO_{X,P}$ and $w_1,\ldots,w_n\in {\mathbf Q}_{>0}$ 
such that the image of $f$ in $\cO_{X,P}$ can be written as $\sum_uc_ux^u$, where the sum is over the set of those
$u=(u_1,\ldots,u_n)\in{\mathbf Z}_{\geq 0}^n$ such that $\sum_{i=1}^nu_iw_i=1$ (where we put $x^u=x_1^{u_1}\cdots x_n^{u_n}$).
We say that $f$ has \emph{weighted homogeneous singularities} if it is weighted homogeneous at every point $P\in Z$.

\begin{rmk}\label{rmk_quasihomog}
Note that since we work in the algebraic setting, in the above definition we require our local coordinates to be algebraic. If we only ask that these
are holomorphic local coordinates, then the condition is equivalent, by a famous result of K.~Saito \cite[Main~Theorem]{ksaito} to the fact that $f$ lies in its Jacobian ideal (whose definition is recalled before Lemma~\ref{sqhomog} below); in this case
one says that $f$ is \emph{quasi-homogeneous}. 
\end{rmk}

\begin{proof}[Proof of Theorem~\ref{thm_quasi_homog}]
Since $F_k\cO_X(*Z)\subseteq P_k\cO_X(*Z)$, we only need to show that $\tfrac{1}{f^{k+1}}\in\cD_X\cdot F_k\cO_X(*Z)$. 
This clearly holds outside of the singular locus of $Z$, hence we only need to focus on the singular points.
Since $Z$ has isolated singularities, after covering $X$ by suitable open subsets, we may and will assume
that we have $P\in Z$ such that $Z\smallsetminus\{P\}$ is smooth.

The key ingredient is Saito's description for the Hodge filtration on $\cO_X(*Z)$ in the case of weighted homogeneous isolated singularities. 
We use the notation introduced at the beginning of this section.
After possibly replacing $X$ with an open neighborhood of $P$, we may assume that $x_1,\ldots,x_n\in\cO_X(X)$ and they give an algebraic system of coordinates on $X$
(that is, $dx_1,\ldots,dx_n$ trivialize the cotangent bundle). We denote by $\partial_{x_1},\ldots,\partial_{x_n}$ the corresponding derivations on $\cO_X$.

For every $u\in{\mathbf Z}_{\geq 0}$,
we put $\rho(u):=\sum_{i=1}^n(u_i+1)w_i$. With this notation, it is shown in \cite[Theorem~0.7]{Saito-HF} that 
$$
\frac{x^u}{f^{k+1}}\in F_k\cO_X(*Z)\quad\text{if}\quad \rho(u)\geq k+1.
$$
In particular, this formula shows that if $\frm_P=(x_1,\ldots,x_n)$ is the ideal defining $P$, then
$$\frm_P^{\ell}\cdot\frac{1}{f^{k+1}}\subseteq F_k\cO_X(*Z)\quad\text{for}\quad \ell\gg 0$$
(of course, this also follows directly from the fact that $Z\smallsetminus\{P\}$ is smooth).
We see that we get the assertion in the theorem if we can show that for every $u\in {\mathbf Z}_{\geq 0}^n$,
if $\tfrac{x^{u+e_i}}{f^{k+1}}\in \cD_X\cdot F_k\cO_X(*Z)$ for $1\leq i\leq n$, then also $\tfrac{x^{u}}{f^{k+1}}\in \cD_X\cdot F_k\cO_X(*Z)$
(here we denote by $e_1,\ldots,e_n$ the standard basis of ${\mathbf Z}^n$).

We may assume that $\rho(u)<k+1$, since otherwise $\tfrac{x^{u}}{f^{k+1}}\in F_k\cO_X(*Z)$ by Saito's result. 
Since $\tfrac{x^{u+e_i}}{f^{k+1}}\in \cD_X\cdot F_k\cO_X(*Z)$ for every $i$, it follows that also
$$\sum_{i=1}^nw_i\partial_{x_i}\cdot \tfrac{x^{u+e_i}}{f^{k+1}}\in \cD_X\cdot F_k\cO_X(*Z).$$
Our assumption on $f$ implies $\sum_{i=1}^nw_ix_i\partial_{x_i}(f)=f$ by Euler's formula, hence
$$\sum_{i=1}^nw_i\partial_{x_i}\cdot \tfrac{x^{u+e_i}}{f^{k+1}}=\sum_{i=1}^nw_i(u_i+1)\frac{x^u}{f^{k+1}}-(k+1)\frac{x^u}{f^{k+2}}\cdot\sum_{i=1}^nw_ix_i\partial_{x_i}(f)=
\big(\rho(u)-(k+1)\big)\frac{x^u}{f^{k+1}}.$$
Since $\rho(u)-(k+1)\neq 0$, we conclude that $\tfrac{x^u}{f^{k+1}}\in \cD_X\cdot F_k\cO_X(*Z)$. This completes the proof of the theorem.
\end{proof}

In particular, by taking $k=0$ and $n\geq 3$ in Theorem~\ref{thm_quasi_homog} and using also Theorem~\ref{thm_intro}, we obtain
the following result due to Bitoun and Schedler, see \cite[Theorem~1.29]{BS}.

\begin{cor}
If $X$ is a smooth complex algebraic variety of dimension $n\geq 3$ and $Z$ is
a hypersurface in $X$ 
defined by $f\in\cO_X(X)$, and $P\in Z$ is such that $Z\smallsetminus\{P\}$ is smooth and $f$ is weighted homogeneous at $P$, then
$$\ell\big(\cD_X\cdot\tfrac{1}{f}\big)= g_P(Z)+2.$$
\end{cor}

\section{A counterexample}
In this section we give a counterexample to the Bitoun-Schedler conjecture. More precisely, we prove the following

\begin{prop}\label{prop_example}
If $f=x^4+y^4+z^4+xy^2z^2\in\C[x,y,z]$ and $X$ is an open neighborhood of $0$ in $\A^3$ such the hypersurface $Z$ defined by $f$ in $X$
has the property that $Z\smallsetminus\{0\}$ is smooth, then
$\ell\big(\cD_X\cdot\tfrac{1}{f}\big)>g_0(Z)+2$.
\end{prop}

\begin{rmk}
An easy computation shows that the zero-locus of $\left(f,\tfrac{\partial f}{\partial x}, \tfrac{\partial f}{\partial y}, \tfrac{\partial f}{\partial z}\right)$ in $\A^3$
is just the origin. This implies that we could take $X=\A^3$ in the above proposition. However, this fact is not important for us.
\end{rmk}

Before giving the proof of the proposition, we need a few preliminary results. Recall that if $f\in\C[x_1,\ldots,x_n]$, the \emph{Jacobian ideal} ${\rm Jac}(f)$ is the ideal
generated by $\tfrac{\partial f}{\partial x_1},\ldots,\tfrac{\partial f}{\partial x_n}$.

\begin{lem}\label{sqhomog} 
Let $f\in\C[x_1,\ldots , x_n]$ be such that $f = g+h$, where $g$ is homogeneous of degree $d\geq 3$, with an isolated singularity at the origin, and 
$h$ is homogeneous of degree $d+1$. If $h\notin {\rm Jac}(g)$, then $f\notin {\rm Jac}(f)$ at $0$.
\end{lem}

\begin{proof}
Clearly, is enough to show that we have
$f\notin \left(\tfrac{\partial f}{\partial x_1},\ldots,\tfrac{\partial f}{\partial x_n}\right)\C\llbracket x_1,\ldots,x_n\rrbracket$.
For $P\in \C\llbracket x_1,\ldots,x_n\rrbracket$, we write $P = P_0+P_1+\ldots$, where $P_i$ is homogeneous of degree $i$. Since $g$ is homogeneous, with an isolated singularity at $0$, 
it follows that $\frac{\partial g}{\partial x_1}, \ldots, \frac{\partial g}{\partial x_n}$ form a regular sequence.

Arguing by contradiction, let us suppose that 
\begin{equation}\label{eq_sqhomog}
f = A^{(1)}\frac{\partial f}{\partial x_1} + \ldots +A^{(n)}\frac{\partial f}{\partial x_n},\quad\text{for some}\quad A^{(1)},\ldots,A^{(n)}\in\C\llbracket x_1,\ldots, x_n\rrbracket.
\end{equation} 
By considering the equality of degree $d-1$ components, we obtain that $A^{(1)}_0 = \cdots = A^{(n)}_0 = 0$, since $\frac{\partial g}{\partial x_1}, \ldots, \frac{\partial g}{\partial x_n}$ are linearly independent over $\C$. By considering the equality of degree $d$ components in (\ref{eq_sqhomog}), we get
 $$g = A^{(1)}_1\frac{\partial g}{\partial x_1}+ \cdots+ A^{(n)}_1\frac{\partial g}{\partial x_n}.$$ 
 Note that since $\frac{\partial g}{\partial x_1}, \ldots, \frac{\partial g}{\partial x_n}$ form a regular sequence 
 of homogeneous polynomials of degree $d-1\geq 2$, there are no nontrivial linear relations 
 between $\frac{\partial g}{\partial x_1}, \ldots, \frac{\partial g}{\partial x_n}$. Euler's relation thus implies 
 $$A^{(i)}_1 = \frac{x_i}{d}\quad\text{for}\quad 1\leq i\leq n.$$ 
 Finally, consider the equality of degree $d+1$ components in (\ref{eq_sqhomog}):
 $$ h = \sum_{i=1}^n\frac{x_i}{d} \cdot \frac{\partial h}{\partial x_i} + \sum_{i=1}^nA^{(i)}_2\frac{\partial g}{\partial x_i}= \frac{d+1}{d} h + \sum_{i=1}^nA^{(i)}_2\frac{\partial g}{\partial x_i}.$$ This gives $h\in {\rm Jac}(g)\cdot\C\llbracket x_1,\ldots,x_n\rrbracket$
and since the homomorphism 
 $$\C[x_1,\ldots,x_n]_{(x_1,\ldots,x_n)}\to \C\llbracket x_1,\ldots,x_n\rrbracket$$ is faithfully flat,
we conclude that $h\in {\rm Jac}(g)$ at $0$. Using the fact that both $h$ and ${\rm Jac}(g)$ are homogeneous, we obtain
$h\in {\rm Jac}(g)$, a contradiction.
\end{proof}

\begin{example}\label{examplesqhomog} 
Let $f= x^4+y^4+z^4+xy^2z^2\in\C[x,y,z]$. Since it is clear that $xy^2z^2\notin (x^3,y^3,z^3)$, it follows from Lemma~\ref{sqhomog} that $f\not\in {\rm Jac}(f)$ at $0$; in particular, $f$ is not weighted homogeneous at $0$ (see Remark~\ref{rmk_quasihomog}).
\end{example}

\begin{rmk}\label{rmk_ordinary}
Suppose that $f=g+h\in \C[x_1,\ldots,x_n]$, with $n\geq 2$, $g$ homogeneous of degree $n+1$, with an isolated singularity at $0$, and $h\in (x_1,\ldots,x_n)^{n+2}$.
In this case the hypersurface $Z$ defined by $f$ has an \emph{ordinary singularity} at $0$: this means that the projectivized tangent cone of $Z$ at $0$ is smooth
(in our case, this is the hypersurface defined by $g$ in ${\mathbf P}^{n-1}$). The blow-up $\pi\colon Y\to \A^n$ of $\A^n$ at $0$ has the property that
$\pi^*Z=\widetilde{Z}+(n+1)E$, where $Z$ is the strict transform of $Z$ and $E$ is the exceptional divisor. The ordinarity condition is equivalent to the fact that
$\widetilde{Z}\cap E$ is smooth, in which case we see that $\pi$ gives a log resolution of $(\A^n,Z)$ in a neighborhood of $0$ (in particular, $0$ is an isolated singularity of $Z$).
Since $K_{Y/\A^n}=(n-1)E$, an easy (and well-known) computation gives $I_0(Z)=(x_1,\ldots,x_n)$ and ${\rm adj}(Z)=(x_1,\ldots,x_n)^2$ in a neighborhood of $0$.
\end{rmk}

\begin{lem}\label{lem_containment}
Let $f= g+h\in\C[x_1,\ldots, x_n]$, where $n\geq 2$, with $g$ homogeneous of degree $n+1$, with an isolated singularity at 0, and $h\in (x_1,\ldots,x_n)^{n+2}$. 
If $f\notin {\rm Jac}(f)$ at $0$ and $Z$ is the hypersurface defined by $f$, then 
$$\frac{1}{f}\notin F_1\cD_X\cdot \left(I_0(Z)\frac{1}{f}\right)\quad\text{at}\quad 0.$$
\end{lem}

\begin{proof}
By Remark~\ref{rmk_ordinary}, we have $I_0(Z)=(x_1,\ldots,x_n)$ in a neighborhood of $0$. Of course, it is enough to show that we have
$$\frac{1}{f}\notin F_1\cD_X\cdot \left((x_1,\ldots,x_n)\frac{1}{f}\right)\quad\text{in}\quad \C\llbracket x_1,\ldots,x_n\rrbracket.$$
Arguing by contradiction, let us suppose that we can write
$$\frac{1}{f} = \sum_{i=1}^np_i(x)\frac{x_i}{f} + \sum_{i,j}^nq_{i,j}(x)\frac{\partial}{\partial x_j} \left(\frac{x_i}{f}\right),
\quad\text{for some}\quad p_i,q_{i,j}\in\C\llbracket x_1,\ldots,x_n\rrbracket.$$ 
Equivalently, we have
 \begin{equation}\label{equality1}
 f\left(1-\sum_{i=1}^np_i(x)x_i - \sum_{i=1}^nq_{i,i}(x)\right) = -\sum_{i,j}^nq_{i,j}(x)x_i\frac{\partial f}{\partial x_j}.
 \end{equation} 
 By comparing the homogeneous components of degree $n+1$, we get
 \begin{equation*}\label{equality2} 
 g\left(1-\sum_{i=1}^nq_{i,i}(0)\right) = -\sum_{i,j}^nq_{i,j}(0)x_i\frac{\partial g}{\partial x_j}.
 \end{equation*} 
 If we put $a_{i,j} = q_{i,j}(0)$ and $b = 1-\sum_i a_{i,i}$, then the above equality becomes 
 $$bg=- \sum_{i,j}^na_{i,j}x_i\frac{\partial g}{\partial x_j}.$$ 
 Since $g$ is homogeneous of degree $n+1$, using Euler's formula we obtain  
 $$\sum_{j=1}^n l_j(x)\frac{\partial g}{\partial x_j} = 0,$$ where $l_j(x) = \frac{b}{n+1}x_j + \sum_{i=1}^na_{i,j}x_i$. 
 Note that $\frac{\partial g}{\partial x_1}, \ldots , \frac{\partial g}{\partial x_n}$ satisfy no nontrivial linear relation: indeed, 
 they form a regular sequence of homogeneous polynomials of degree $n\geq 2$,
 since $g$ has an isolated singularity at $0$. Therefore $l_j(x) = 0$ for all $j$. It follows that we have  $a_{i,j} = 0$ for $i\neq j$ and $a_{j,j} = -\frac{b}{n+1}$ for all $j$. 
 Since $b=1-\sum_ja_{j,j}$, we conclude that  
 $b= n+1$ and $a_{j,j} = -1$ for all $j$. In particular, we have
 $$1-\sum_{i=1}^np_i(x)x_i - \sum_{i=1}^nq_{i,i}(x) \equiv n+1 \quad \big({\rm mod}\,\,(x_1,\ldots x_n)\big).$$ 
 Using (\ref{equality1}), we conclude that $f\in {\rm Jac}(f)\cdot\C\llbracket x_1,\ldots,x_n\rrbracket$.
 Since the homomorphism 
 $$\C[x_1,\ldots,x_n]_{(x_1,\ldots,x_n)}\to \C\llbracket x_1,\ldots,x_n\rrbracket$$ is faithfully flat,
 we conclude that $f\in {\rm Jac}(f)$ at $0$, a contradiction.
\end{proof}

We can now show that we get a counterexample to the Bitoun-Schedler conjecture.

\begin{proof}[Proof of Proposition~\ref{prop_example}]
Recall that $f$ has an ordinary singularity at $0$ (see Remark~\ref{rmk_ordinary}). In particular, the singularity at $0$ is isolated and we may choose $X$
as in the statement of the proposition. 
We freely use the notation in Section \ref{sectionmainresult}. In order to simplify the notation in what follows, instead of working in $\cH^1_Z(\cO_X)$, we will 
mostly work in $\cO_X(*Z)=\cO_X[1/f]$. 
We will make essential use of the Hodge filtration on the 
mixed Hodge module $\cO_X(*Z)$ and on its submodule $\widetilde{M}_f$, the inverse image of $M_f\subseteq \cH_Z^1(\cO_X)$.

Since
$\cO_X(*Z)/\widetilde{M}_f=\cH_Z^1(\cO_X)/M_f$ is supported on $0\in\A^n$, it is of the form $i_+H$, where $H$ is a mixed Hodge module on the point $0$ (that is, a mixed Hodge structure) and 
$i\colon\{0\}\hookrightarrow \A^n$ is the inclusion. The quotient $N_f=\cD_X\cdot\frac{1}{f} /\widetilde{M}_f$ is a 
$\cD_X$-submodule of $i_+H$, and therefore is of the form $i_+H'$ for some vector subspace $H'\subseteq H$. 
The length of $N_f$ as a $\cD_X$-module is equal to $\dim_{\C}(H')$.
Note that $H'$ is not necessarily a mixed Hodge
substructure of $H$. However, it is convenient to put
$$F_kN_f:=F_k(i_+H)\cap N_f\quad\text{for all}\quad k\geq 0.$$

We recall that $F_k(i_+H)=0$ for $k<0$ (since the same property holds for $\cH^1_Z(\cO_X)$) and 
the standard convention is that if we write $i_+H=H\otimes_{\C}\C[\partial_x,\partial_y,\partial_z]$, then 
$$F_k(i_+H)=\bigoplus_j\bigoplus_{a+b+c=j}F_{k-3-j}H\otimes\partial_x^a\partial_y^b\partial_z^c$$
(see \cite{Saito-HF}*{(1.5.3)}). This implies that
\begin{equation}\label{eq_prop_example}
F_1(i_+H)\cap\big(\cD_X\cdot F_0(i_+H)\big)\subseteq F_1\cD_X\cdot F_0(i_+H).
\end{equation}

We have seen in the proof of Theorem \ref{thm_intro} that $F_0(i_+H)\subseteq N_f$ and 
$$F_0(i_+H)=F_0\cH_Z^1(\cO_X)/F_0M_f=\frac{I_0(Z)\tfrac{1}{f}}{{\rm adj}(Z)\tfrac{1}{f}}$$
is a subspace of $H'\otimes 1$ of dimension $g_0(Z)$. The strict inequality $\ell\big(\cD_X\cdot\tfrac{1}{f}\big)>g_0(Z)+2$
is equivalent to having $F_0(i_+H)\neq H'\otimes 1$. For this, it is enough to show that there is an element $u\in F_1N_f
\smallsetminus F_1\cD_X\cdot F_0(i_+H)$: indeed, if $F_0(i_+H)=H'\otimes 1$, then it follows from (\ref{eq_prop_example}) 
that 
$$F_1N_f\subseteq F_1(i_+H)\cap \cD_X\cdot (H'\otimes 1)\subseteq F_1\cD_X\cdot F_0(i_+H).$$

Since $f$ has an ordinary singularity at $0$, it is semi-quasi-homogeneous in the sense of 
\cite[Section~5]{Saito-HF}. This allows us to compute the Hodge filtration on $\widetilde{M}_f$ and $\cO_X(*Z)$, as follows. 
If we give each variable weight $1/4$ (so that $x^4+y^4+z^4$ has all terms of weight $1$) and if we denote by 
$P_{\geq k+1}$ (resp. $P_{>k+1}$) the linear span of the quotients $\tfrac{x^ay^bz^c}{f^{k+1}}$ with
 $\tfrac{1}{4}(a+b+c+3)\geq k+1$ (resp. $>k+1$),
it follows from \cite[Theorem~0.9]{Saito-HF} that for every $p$, in some neighborhood of $0$ we have
$$F_p\cO_X(*Z)=\sum_{k=0}^pF_{p-k}\cD_X\cdot P_{\geq k+1}\quad\text{and}\quad F_p\widetilde{M}_f=\sum_{k=0}^pF_{p-k}\cD_X\cdot P_{>k+1}.$$
For $p=0$, we recover the formulas computed in a different way in Remark~\ref{rmk_ordinary}:
$$F_0\cO_X(*Z)=I_0(Z)\frac{1}{f}=(x,y,z)\frac{1}{f}\quad\text{and}\quad F_0\widetilde{M}_f={\rm adj}(Z)\frac{1}{f}=(x,y,z)^2\frac{1}{f}.$$
For the next step in the two filtrations, if we define $J_1(Z)$ and $K_1(Z)$ by
$$F_1\cD_X\cdot  \left(I_0(Z)\frac{1}{f}\right)= J_1(Z)\frac{1}{f^2}\quad \text{ and }\quad F_1\cD_X\cdot \left({\rm adj}(Z)\frac{1}{f}\right) = 
K_1(Z)\frac{1}{f^2},$$
we get
\begin{equation}\label{eq10_counterexample}
F_1\cO_X(*Z)=\big(J_1(Z)+(x,y,z)^5\big)\frac{1}{f^2}\quad\text{and}\quad F_1\widetilde{M}_f=\big(K_1(Z)+(x,y,z)^6\big)\frac{1}{f^2}
\end{equation}
(actually, one can show that $(x,y,z)^6\subseteq K_1(Z)$, but we will not use this fact).

We claim that the image $u\in\cO_X(*Z)/\widetilde{M}_f$ of $v=\tfrac{xy^2z^2}{f^2}$ lies in $F_1N_f
\smallsetminus F_1\cD_X\cdot F_0(i_+H)$. As we have seen, this is enough to conclude the proof. 
Note first that $v\in \cD_X\cdot\tfrac{1}{f}$: indeed, an easy computation gives
\begin{equation}\label{eq11_counterexample}
\tfrac{xy^2z^2}{f^2}=-(\partial_xx+\partial_yy+\partial_zz+1)\cdot\tfrac{1}{f}.
\end{equation}
Since $xy^2z^2\in (x,y,z)^5$, it follows from (\ref{eq10_counterexample}) that $v\in F_1\cO_X(*Z)$, and we thus conclude that $u\in F_1N_f$. 
On the other hand, if $u\in F_1\cD_X\cdot F_0(i_+H)$, then using the fact that the inclusion $\widetilde{M}_f\hookrightarrow\cO_X(*Z)$
is strict with respect to the Hodge filtration and (\ref{eq10_counterexample}), we obtain
$$v\in F_1\cO_X(*Z)\cap \left(F_1\cD_X\cdot (x,y,z)\frac{1}{f}+\widetilde{M}_f\right)\subseteq
\left(F_1\cD_X\cdot (x,y,z)\frac{1}{f}\right)+F_1\widetilde{M}_f$$
$$=J_1(Z)\frac{1}{f^2}+\big(K_1(Z)+(x,y,z)^6\big)\frac{1}{f^2}=\big(J_1(Z)+(x,y,z)^6\big)\frac{1}{f^2}.$$
This contradicts Lemma~\label{final_lemma} below and thus the proof is complete.
\end{proof}

\begin{lem}\label{lemcounterexample} 
With the notation in the proof of Proposition~\ref{prop_example}, we have  $xy^2z^2\notin J_1(Z)+(x,y,z)^6$. 
\end{lem}

\begin{proof}
Note first that by Example~\ref{examplesqhomog}, we know that $f\notin {\rm Jac}(f)$ at $0$. Therefore we may apply Lemma~\ref{lem_containment}
to conclude that $f\not\in J_1(Z)$. On the other hand, the relation (\ref{eq11_counterexample}) implies
$$\frac{f+xy^2z^2}{f^2}\in F_1\cD_X\cdot \left((x,y,z)\frac{1}{f}\right)=J_1(Z)\frac{1}{f^2},$$
hence $f+xy^2z^2\in J_1(Z)$, so that $xy^2z^2\notin J_1(Z)$. Therefore we are done if we show that $(x,y,z)^6\subseteq J_1(Z)$. 

Since $(x,y,z)\tfrac{1}{f}\subseteq J_1(Z)\tfrac{1}{f^2}$, it follows that
\begin{equation}\label{eq_final1}
xf, yf,zf\in J_1(Z).
\end{equation}
We have already seen that $f+xy^2z^2\in J_1(Z)$, and thus $x(f+xy^2z^2)-xf$, $y(f+xy^2z^2)-yf, z(f+xy^2z^2)-zf\in J_1(Z)$, hence
\begin{equation}\label{eq_final2}
x^2y^2z^2, xy^3z^2, xy^2z^3\in J_1(Z).
\end{equation}
By considering  $x\partial_x\tfrac{y}{f}$, $x\partial_x\tfrac{z}{f}$, $y\partial_y\tfrac{z}{f}$, $z\partial_z\tfrac{y}{f}$, $y\partial_y\tfrac{x}{f}$, $z\partial_z\tfrac{x}{f}$, and using
the terms in (\ref{eq_final2}), we obtain
\begin{equation}\label{eq_final4}
x^4y, x^4z, y^4z, yz^4, xy^4, xz^4\in J_1(Z).
\end{equation}
By considering $z\partial_y\tfrac{z}{f}$ and $y\partial_z\frac{y}{f}$
and using (\ref{eq_final4}), we obtain 
\begin{equation}\label{eq_final10}
y^3z^2, y^2z^3\in J_1(Z). 
\end{equation}
By considering $\partial_x\tfrac{y}{f}$ and $\partial_x\tfrac{z}{f}$, together with (\ref{eq_final10}), we get
\begin{equation}\label{eq_final11}
x^3y, x^3z\in J_1(Z)
\end{equation}
 and by considering the terms in (\ref{eq_final11}) and $x\partial_y\tfrac{x}{f}$, $x\partial_z\tfrac{x}{f}$, 
we obtain 
$$
x^2y^3,x^2z^3\in J_1(Z).
$$
Using again the terms in (\ref{eq_final1}), as well as the ones in (\ref{eq_final2}) and (\ref{eq_final4}), we obtain
$$
x^5, y^5, z^5\in J_1(Z). 
$$
It is now straightforward to see that by putting together all the above terms, we get $(x,y,z)^6\subseteq J_1(Z)$, completing the proof of the lemma.
\end{proof}

\end{document}